\documentclass[11pt]{article}
\usepackage{amsfonts}
\usepackage{mathrsfs}

\usepackage[latin1]{inputenc}
\usepackage{amsmath,amssymb}
\usepackage{latexsym}
\usepackage[active]{srcltx}
\usepackage[
bookmarks=true,         
bookmarksnumbered=true, 
colorlinks=true, pdfstartview=FitV, linkcolor=blue, citecolor=blue,
urlcolor=blue]{hyperref}

\newtheorem{theorem}{Theorem}[section]

\newtheorem{lemma}[theorem]{Lemma}

\newtheorem{remark}[theorem]{Remark}
\numberwithin{equation}{section}
\def\R{{\mathbb R}}
\def\E{{{\mathbb E}\,}}

\def\N{{\mathbb N}}

\def\square{{\vcenter{\vbox{\hrule height.3pt
        \hbox{\vrule width.3pt height5pt \kern5pt
           \vrule width.3pt}
        \hrule height.3pt}}}}

\def\tlint{{- \kern-0.85em \int \kern-0.2em}}
\def\dlint{{- \kern-1.05em \int \kern-0.4em}}

\def \eref#1{\hbox{(\ref{#1})}}

\def \eref#1{\hbox{(\ref{#1})}}

\newenvironment{proof}[1][Proof]{\noindent\textit{#1.} }{\hfill \rule{0.5em}{0.5em}}

\begin{document}

\title{Derivatives of local times for some Gaussian fields II}
\date{\today}
\author{Minhao Hong and Fangjun Xu\thanks{M. Hong is partially supported by National Natural Science Foundation of China (Grant No.11871219). F. Xu is partially supported by National Natural Science Foundation of China (Grant No.11871219, No.11871220).} \\
}
\maketitle
\begin{abstract}
\noindent  Given a $(2,d)$-Gaussian field
\[
Z=\big\{ Z(t,s)= X^{H_1}_t -\widetilde{X}^{H_2}_s, s,t \ge 0\big\},
\]
where $X^{H_1}$ and $\widetilde{X}^{H_2}$ are independent $d$-dimensional centered Gaussian processes satisfying certain properties, we will give the necessary condition for existence of derivatives of the local time of $Z$.


\end{abstract}

\section{Introduction}
Local times for Gaussian processes or fields are important in the probability theory. Recently, their derivatives received much attention, see, for example, \cite{jnp, hx} and references therein. In \cite{hx},  we consider derivatives of local time for a $(2,d)$-Gaussian field
\[
Z=\big\{ Z(t,s)= X^{H_1}_t -\widetilde{X}^{H_2}_s, s,t \ge 0\big\},
\]
where $X^{H_1}$ and $\widetilde{X}^{H_2}$ are independent processes from a class of $d$-dimensional centered Gaussian processes satisfying certain local nondeterminism property.  A sufficient condition for existence of derivatives of the local time of $Z$ was given. Then, under the condition, derivatives of the local time are shown to be H\"{o}lder continuous in both time and space variables. Moreover, under some mild assumption, the sufficient condition is necessary for existence of derivatives of the local time at the origin. However, when the location is not  the origin, the necessity of the condition is still open. To the best of our knowledge, for local times of Gaussian processes or fields, in most cases, one just gave sufficient conditions for existence of local times and their derivatives. Only in the case where the location is the origin, there are a few papers showing the necessity of the sufficient condition, see, \cite{no, wx} for local times, \cite{jnp} for the first derivative of local time and \cite{ghx, hx} for both local times and their derivatives. Moreover, the Gaussian processes in \cite{hx} can be Bifractional Brownian motions and Subfractional Brownian motions, which are more general than the fractional Brownian motions in \cite{no,wx,ghx,jnp}.

In this paper, when the location is not the origin, we will show that,  if the Gaussian processes in \cite{hx} satisfying some additional properties,  then the sufficient condition $\frac{H_1H_2}{H_1+H_2}(2|\mathbf{k}|+d)<1$ is necessary. Moreover, when the location is the origin and under certain additional assumption, we could give a shorter proof for the necessity of the sufficient condition $\frac{H_1H_2}{H_1+H_2}(2|\mathbf{k}|+d)<1$. To make our results as general as possible, we only consider the necessary condition for existence of derivatives of local times here and would pose less restrictions on the Gaussian processes.  For $H\in (0,1)$,  let $X^H= \{ X^H_t:\, t\geq 0\} $ be a $d$-dimensional centered Gaussian stochastic process whose components $X^{H,\ell} (1\le \ell \le d)$ are independent and identically distributed, and satisfy the following property:

{\bf (P1)}    {\it Bounds on the second moment of increments}: for any $T>0$, there exist two positive constants $\kappa_{T,1,H}\leq \kappa_{T,2,H}$ depending only on $T$ and $H$, such that for any $0\leq s<t<T$,
\begin{align*}
\kappa_{T,1,H}(t-s)^{2H}\leq \E\Big[(X^{H,\ell}_{t} -X^{H,\ell}_{s})^2\Big]\leq \kappa_{T,2,H}(t-s)^{2H}.
\end{align*}
Let $G^d_{1}$ be the class of all such $d$-dimensional centered Gaussian processes and $G^d_{1,2}$ the class of $d$-dimensional centered Gaussian processes in $G^d_{1}$ possessing the additional property:

{\bf (P2)}    {\it Bounds on the covariance of increments on disjoint intervals}:   there exists a nonnegative decreasing function $\beta(\gamma): (1,\infty) \rightarrow \mathbb{R}$ with $\lim\limits_{\gamma\to\infty}\beta(\gamma)=0$, such that, for any $0<s<t<T$ with $\frac{t-s}{s}\leq \frac{1}{\gamma}$,
 \[
\Big|\E\big[(X^{H,\ell}_{t}-X^{H,\ell}_{s})X^{H,\ell}_{s}\big]\Big|\leq  \beta(\gamma)\, \Big[\E\big(X^{H,\ell}_{t}-X^{H,\ell}_{s}\big)^2\Big]^{\frac{1}{2}}\Big[\E\big(X^{H,\ell}_{s}\big)^2 \Big]^{\frac{1}{2}}.
\]
According to results in \cite{bgt, hv, sxy}, we can easily see that the following $d$-dimensional Gaussian processes are in $G^d_{1,2}$.

\noindent
(i) {\it Bifractional Brownian motion (bi-fBm)}. The covariance function for components of this process is given by
\[
\E(X^{H,\ell}_t X^{H,\ell}_s)=2^{-K_0}\big[(t^{2H_0}+s^{2H_0})^{K_0}-|t-s|^{2H_0K_0}\big],
\]
where $H_0\in(0,1)$ and $K_0\in(0,1]$.  Here $K_0=1$ gives the fractional Brownian motion (fBm) with Hurst parameter $H=H_0$.

\noindent
(ii) {\it Subfractional Brownian motion (sub-fBm)}.  The covariance function for components of this process  is given by
\[
\E(X^{H,\ell}_t X^{H,\ell}_s)=t^{2H}+s^{2H}-\frac{1}{2}\big[(t+s)^{2H}+|t-s|^{2H}\big],
\]
where $H\in(0,1)$.

Let $X^{H_1}$ and $\widetilde{X}^{H_2}$ be independent Gaussian processes in $G^d_1$ with parameters $H_1, H_2\in (0,1)$, respectively. Then
\begin{align}
Z=\big\{Z(t,s)= X^{H_1}_t -\widetilde{X}^{H_2}_s, s,t \ge 0\big\}  \label{gf}
\end{align}
is a $(2,d)$-Gaussian field.

For any $\varepsilon>0$ and the multi-index $\mathbf{k}=(k_1,\cdots, k_d)$ with all $k_i$ being nonnegative integers, let
\[
p^{(\mathbf{k})}_{\varepsilon}(x)=\frac{\partial^\mathbf{k}}{\partial x^{k_1}_1\cdots \partial x^{k_d}_d}p_{\varepsilon}(x) =\frac{\iota^{|\mathbf{k}|}}{(2\pi)^d}\int_{\R^d} \Big(\prod^d_{i=1}y^{k_i}_i\Big)\, e^{\iota y\cdot x}e^{-\frac{\varepsilon|y|^2}{2}}\, dy,
\]
where $p_{\varepsilon}(x)=\frac{1}{(2\pi \varepsilon)^{\frac{d}{2}}} e^{-\frac{|x|^2}{2\varepsilon}}$ and $|\mathbf{k}|=\sum\limits^d_{i=1}k_i$.

For any $T>0$ and $x\in\R^d$, if
\begin{align}\label{epsilon}
L^{(\mathbf{k})}_{\varepsilon}(T,x):=\int^T_0\int^T_0 p^{(\mathbf{k})}_{\varepsilon}(X^{H_1}_t-\widetilde{X}^{H_2}_s+x)\, ds\, dt
\end{align}
converges to some random variable in $L^2$  when $\varepsilon\downarrow 0$, we denote the limit by $L^{(\mathbf{k})}(T,x)$ and call it  the $\mathbf{k}$-th derivative of local time for the $(2,d)$-Gaussian field $Z$.  If it exists,  $L^{(\mathbf{k})}(T,x)$ admits the following $L^2$-representation
\begin{align} \label{dlt}
L^{(\mathbf{k})}(T,x)=\int^T_0\int^T_0 \delta^{({\bf k})}(X^{H_1}_t-\widetilde{X}^{H_2}_s+x)\, ds\, dt.
\end{align}

The following are main results of this paper.
\begin{theorem}\label{thm1}
Assume that $X^{H_1}=\{X^{H_1}_t:\, t\geq 0\}$ and $\widetilde{X}^{H_2}=\{\widetilde{X}^{H_2}_t:\, t\geq 0\}$ are two independent Gaussian processes in $G^d_{1,2}$ with parameters $H_1, H_2\in(0,1)$, respectively.  For any $x\neq 0$, if $\frac{H_1H_2}{H_1+H_2}(2|\mathbf{k}|+d)\geq 1$, then there exist positive constants $c_1$ and $c_2$ such that
\begin{align*}
\liminf_{\varepsilon\downarrow 0}\frac{\E[|L^{(\mathbf{k})}_{\varepsilon}(T,x)|^2]}{h^{d,|{\bf k}|}_{H_1,H_2}(\varepsilon)}\geq c_1e^{-c_2|x|^2},
\end{align*}
where
\begin{align} \label{rate}
h^{d,|{\bf k}|}_{H_1,H_2}(\varepsilon)=
\left\{\begin{array}{ll}
\varepsilon^{\frac{H_1+H_2}{2H_1H_2}-\frac{d}{2}-|\mathbf{k}|} &   \text{if}\; \frac{H_1H_2}{H_1+H_2}(2|\mathbf{k}|+d)>1\\  \\
\ln(1+\varepsilon^{-\frac{1}{2}}) &   \text{if}\; \frac{H_1H_2}{H_1+H_2}(2|\mathbf{k}|+d)=1.
\end{array} \right.
\end{align}
\end{theorem}

\begin{theorem}\label{thm2}
Assume that $X^{H_1}=\{X^{H_1}_t:\, t\geq 0\}$ and $\widetilde{X}^{H_2}=\{\widetilde{X}^{H_2}_t:\, t\geq 0\}$ are two independent Gaussian processes in $G^d_{1}$ with parameters $H_1, H_2\in(0,1)$, respectively. We further assume that  (i) $|\mathbf{k}|$ is even or  (ii) $\E[X^{H_1,1}_{t}X^{H_1,1}_{s}]\geq 0$ and  $\E[\widetilde{X}^{H_2,1}_{t}\widetilde{X}^{H_2,1}_{s}]\geq 0$ for any  $0<s,t<T$. Then, if $\frac{H_1H_2}{H_1+H_2}(2|\mathbf{k}|+d)\geq 1$, there exists a positive constant $c_3$ such that
\[
\liminf\limits_{\varepsilon\downarrow 0}\frac{\E[|L^{(\mathbf{k})}_{\varepsilon}(T,0)|^2]}{h^{d,|{\bf k}|}_{H_1,H_2}(\varepsilon)}\geq c_3.
\]
\end{theorem}

\begin{remark} For independent Gaussian processes $X^{H_1}$ and $\widetilde{X}^{H_2}$ in $G^d_{L}$ as defined in \cite{hx}, if they are also in $G^d_{1,2}$,  then, for $x\neq 0$, $L^{(\mathbf{k})}(T,x)$ exists in $L^2$ if and only if $\frac{H_1H_2}{H_1+H_2}(2|\mathbf{k}|+d)<1$.
\end{remark}

\begin{remark} For independent Gaussian processes $X^{H_1}$ and $\widetilde{X}^{H_2}$ in $G^d_{L}$ as defined in \cite{hx}, if they also satisfy the assumptions in Theorem \ref{thm2},  then $L^{(\mathbf{k})}(T,0)$ exists in $L^2$ if and only if $\frac{H_1H_2}{H_1+H_2}(2|\mathbf{k}|+d)<1$. Moreover, the definition of $d$-dimensional bi-fBm implies that its components have nonnegative covariance function. The non-negativity of covariance function for components of sub-fBm follows from \cite{bgt}.
\end{remark}

\begin{remark}
For each $N\in\N$, define the $(N, d)$-Gaussian field
\[
 Z^{N}=\Big\{\sum^{N}_{j=1} X^{j,H_j}_{t_j}:\; t_j\geq 0,\, j=1,\dots,N\Big\},
\]
where $X^{j,H_j}_{t_j}$ are independent $d$-dimensional Gaussian processes in $G^d_{1}$. Replace $Z$ in $L^{(\mathbf{k})}_{\varepsilon}(T,x)$ and $L^{(\mathbf{k})}(T,x)$ by $Z^{N}$ and denote the new terms by $L^{(\mathbf{k})}_{N, \varepsilon}(T,x)$ and $L^{(\mathbf{k})}_{N}(T,x)$, respectively. Using the methodologies developed here, we could obtain similar results for the existence of $L^{(\mathbf{k})}_{N}(T,x)$ in $L^2$. That is,
\begin{enumerate}
\item[(i)] Assume that $X^{j,H_j}_{t_j} (j=1,\dots,N)$ are independent $d$-dimensional Gaussian processes in $G^d_{1,2}$, $x\neq 0$ and $2|\mathbf{k}|+d\geq \sum\limits^N_{j=1} \frac{1}{H_j}$.  Then there exist positive constants $c_4$ and $c_5$ such that
    \[
\liminf\limits_{\varepsilon\downarrow 0}\frac{\E[|L^{(\mathbf{k})}_{N,\varepsilon}(T,x)|^2]}{h^{d,|{\bf k}|}_{H_1,H_2,\dots, H_N}(\varepsilon)}\geq c_4e^{-c_5|x|^2},
\] where $h^{d,|{\bf k}|}_{H_1,H_2,\dots, H_N}(\varepsilon)=
\left\{\begin{array}{ll}
\varepsilon^{\sum\limits^N_{j=1}\frac{1}{2H_j}-\frac{d}{2}-|\mathbf{k}|} &   \text{if}\;\; 2|\mathbf{k}|+d>\sum\limits^N_{j=1} \frac{1}{H_j}\\
\ln(1+\varepsilon^{-\frac{1}{2}}) &   \text{if}\;\; 2|\mathbf{k}|+d=\sum\limits^N_{j=1} \frac{1}{H_j}.
\end{array} \right.$

\item[(ii)] Assume that $X^{j,H_j}_{t_j} (j=1,\dots,N)$ are independent $d$-dimensional Gaussian processes in $G^d_{1}$, $|\mathbf{k}|$ is even or all $X^{j,H_j}$ have nonnegative covariance functions, and $2|\mathbf{k}|+d\geq \sum\limits^N_{j=1} \frac{1}{H_j}$. Then there exists a positive constant $c_6$ such that $
\liminf\limits_{\varepsilon\downarrow 0}\frac{\E[|L^{(\mathbf{k})}_{N,\varepsilon}(T,0)|^2]}{h^{d,|{\bf k}|}_{H_1,H_2,\dots, H_N}(\varepsilon)}\geq c_6.$
\end{enumerate}
The case $N=1$ is easy. It follows from simplified proofs of Theorems \ref{thm1} and \ref{thm2}.  Moreover, our methodologies also work for derivatives of local times of L\'{e}vy processes or fields.
\end{remark}

After some preliminaries in Section 2, Sections 3 and 4 are devoted to the proofs of Theorems \ref{thm1} and \ref{thm2}, respectively. Throughout this paper, if not mentioned otherwise, the letter $c$, with or without a subscript, denotes a generic positive finite constant whose exact value may change from line to line.  For any $x,y\in\R^d$, we use $x\cdot y$ to denote the usual inner product  and $|x|=(\sum\limits^d_{i=1}|x_i|^2)^{1/2}$.  Moreover, we use $\iota$ to denote $\sqrt{-1}$.

\section{Preliminaries}

In this section, we give three lemmas. The first two will be used in the proof of Theorem \ref{thm1} and the last one in the proof of Theorem \ref{thm2}.
\begin{lemma} \label{lma1} Assume that $k\in\N\cup\{0\}$ and $\varepsilon>0$. Then, for any $a,b,c, x\in\R$ with $a>0$, $c>0$ and $\Delta=c+\varepsilon-\frac{(b-\varepsilon)^2}{a+2\varepsilon}>0$, we have
\begin{align} \label{integral}
&\frac{(-1)^k}{2\pi}\int_{\R^2} \exp\Big\{-\frac{1}{2}(y_2^2a+2y_2 y_1b+y_1^2c)-\frac{\varepsilon}{2}((y_1-y_2)^2+y_2^2)+\iota y_1x\Big\}  y^{k}_2 (y_1-y_2)^{k} \, dy \nonumber \\
&=\sum^k_{\ell=0}\sum^{k+\ell}_{m=0:\text{even}}\sum^{2k-m}_{n=0:\text{even}} c_{k,\ell,m,n}  (a+2\varepsilon)^{-\frac{m+1}{2}} (\frac{\varepsilon-b}{a+2\varepsilon})^{k+\ell-m}\Delta^{-(2k-m)-\frac{1-n}{2}} x^{2k-m-n}e^{-\frac{x^2}{2\Delta}},
\end{align}
where
\[
c_{k,\ell,m,n}=(-1)^{\ell-\frac{m+n}{2}} \binom{k}{\ell}  \binom{k+\ell} {m}\binom{2k-m} {n}(m-1)!! (n-1)!!
\]
and we use the convention $0^0=1$ for the case $x=0\in\R$.
\end{lemma}
\begin{proof}  Let $L$ be the left hand side of the equality \eref{integral}. It is easy to show that
\begin{align*}
L&=\frac{(-1)^k}{2\pi}\int_{\R^2} \exp\Big\{-\frac{1}{2}y_2^2(a+2\varepsilon)-y_2 y_1(b-\varepsilon)-\frac{1}{2}y_1^2(c+\varepsilon)+\iota y_1  x\Big\}  y^{k}_2 (y_1-y_2)^{k} \, dy\\
&=\sum^k_{\ell=0}\frac{(-1)^{k+\ell}}{2\pi} \binom{k}{\ell}\int_{\R^2} \exp\Big\{-\frac{1}{2}y_2^2(a+2\varepsilon)-y_2 y_1(b-\varepsilon)-\frac{1}{2}y_1^2(c+\varepsilon)+\iota y_1  x\Big\}  y^{k+\ell}_2 y_1^{k-\ell} \, dy\\
&=\sum^k_{\ell=0}\sum^{k+\ell}_{m=0:\, \text{even}}  \frac{(-1)^{k+\ell}}{\sqrt{2\pi}}\binom{k}{\ell}  \binom{k+\ell} {m}(m-1)!! (a+2\varepsilon)^{-k-\ell-\frac{1-m}{2}} (\varepsilon-b)^{k+\ell-m} \int_{\R} y_1^{2k-m}e^{-\frac{1}{2} y_1^2\Delta+\iota y_1 x}   \, dy_1\\
&=\sum^k_{\ell=0}\sum^{k+\ell}_{m=0:\, \text{even}}\sum^{2k-m}_{n=0:\, \text{even}} c_{k,\ell,m,n} (a+2\varepsilon)^{-\frac{m+1}{2}} (\frac{\varepsilon-b}{a+2\varepsilon})^{k+\ell-m}\Delta^{-(2k-m)-\frac{1-n}{2}} x^{2k-m-n}e^{-\frac{x^2}{2\Delta}}.
\end{align*}
\end{proof}

\begin{lemma} \label{lma2}  Assume that $k\in\N\cup\{0\}$ and $\varepsilon>0$. Then, for any $a_1,b_1,c_1, a_2,b_2,c_2, x\in\R$ with $a_1, a_2, c_1, c_2>0$ and $\Delta'=c_1+c_2+a_2+2b_2+\varepsilon-\frac{(b_1-b_2-a_2-\varepsilon)^2}{a_1+a_2+2\varepsilon}>0$, we have
\begin{align*}
&\frac{(-1)^k}{2\pi}\int_{\R^2} \exp\left\{-\frac{1}{2}\big([y_2^2a_1+2y_2 y_1b_1+y_1^2c_1]+[(y_1-y_2)^2a_2+2(y_1-y_2)y_1b_2+y^2_1c_2]\big)\right\}\\
&\qquad\qquad\times \exp\Big\{-\frac{\varepsilon}{2}((y_1-y_2)^2+y_2^2)+\iota y_1x\Big\}  y^{k}_2 (y_1-y_2)^{k} \, dy\\
&=\sum^k_{\ell=0}\sum^{k+\ell}_{m=0:\text{even}}\sum^{2k-m}_{n=0:\text{even}} c_{k,\ell,m,n} (a_1+a_2+2\varepsilon)^{-\frac{m+1}{2}} (\frac{\varepsilon+b_2+a_2-b_1}{a_1+a_2+2\varepsilon})^{k+\ell-m}(\Delta')^{-(2k-m)-\frac{1-n}{2}} x^{2k-m-n}e^{-\frac{x^2}{2\Delta'}},
\end{align*}
where
\[
c_{k,\ell,m,n}=(-1)^{\ell-\frac{m+n}{2}} \binom{k}{\ell}  \binom{k+\ell} {m}\binom{2k-m} {n}(m-1)!! (n-1)!!
\]
and we use the convention $0^0=1$ for the case $x=0\in\R$.
\end{lemma}
\begin{proof}  Note that the integral in the above statement can be written as
\begin{align*}
\int_{\R^2} \exp\Big\{-\frac{1}{2}y_2^2(a_1+a_2+2\varepsilon)-y_2 y_1(b_1-a_2-b_2-\varepsilon)-\frac{1}{2}y_1^2(c_1+c_2+a_2+2b_2+\varepsilon)+\iota y_1  x\Big\}  y^{k}_2 (y_1-y_2)^{k} \, dy.
\end{align*}
Then the desired result follows from Lemma \ref{lma1}.
\end{proof}

\begin{lemma} \label{lma3}   Assume that $k\in\N\cup\{0\}$ and $\varepsilon>0$. Then, for any $a,b,c\in\R$ with $a,c>0$ and $(a+\varepsilon)(c+\varepsilon)-b^2>0$, we have
\begin{align*}
&\frac{(-1)^k}{2\pi}\int_{\R^2} \exp\Big\{-\frac{1}{2}(y_2^2a+2y_2 y_1b+y_1^2c)-\frac{\varepsilon}{2}(y_2^2+y_1^2)\Big\}  y^{k}_2 y_1^{k} \, dy\\
&\qquad\qquad\qquad\qquad\qquad\qquad=\sum^k_{\ell=0, \text{even}}    \frac{c_{k,\ell} \,b^{k-\ell}}{((a+\varepsilon)(c+\varepsilon)-b^2)^{\frac{2k-\ell+1}{2}}},
\end{align*}
where $c_{k,\ell}=(\ell-1)!!  \binom{k}{\ell} (2k-\ell-1)!!$.
\end{lemma}
\begin{proof} This follows from similar arguments as in the proof of Lemma \ref{lma1}.
\end{proof}

\section{Proof of Theorem \ref{thm1}}

In this section, we give the proof of Theorem \ref{thm1}.

\begin{proof} We divide the proof into several steps.

\noindent
{\bf Step 1.}
Recall the definition of $L^{(\mathbf{k})}_{\varepsilon}(T,x)$ in (\ref{epsilon}). Using Fourier transform,
\begin{align*}
L^{(\mathbf{k})}_{\varepsilon}(T,x)
&=\frac{\iota^{|\mathbf{k}|}}{(2\pi)^d} \int^T_0\int^T_0\int_{\R^d} e^{\iota z\cdot (X^{H_1}_u-\widetilde{X}^{H_2}_v+x)}e^{-\frac{\varepsilon|z|^2}{2}}  \prod^d_{i=1}z^{k_i}_i\, dz\, du\, dv.
\end{align*}
Hence
\begin{align*}
\E[|L^{(\mathbf{k})}_{\varepsilon}(T,x)|^2]
&=\frac{(-1)^{|\mathbf{k}|}}{(2\pi)^{2d}}\int_{[0,T]^4}\int_{\R^{2d}} e^{-\frac{1}{2}\big[\E(z_2\cdot X^{H_1}_{t_2}+z_1\cdot X^{H_1}_{t_1})^2+\E(z_2\cdot \widetilde{X}^{H_2}_{s_2}+z_1\cdot \widetilde{X}^{H_2}_{s_1})^2\big]}\\
&\qquad\qquad \times e^{-\frac{\varepsilon}{2}(|z_2|^2+|z_1|^2)+\iota(z_1+z_2)\cdot x} \prod\limits^d_{i=1}z^{k_i}_{2,i} \prod\limits^d_{i=1}z^{k_i}_{1,i}\, dz_2\, dz_1\, dt\, ds,
\end{align*}
where $z_1=(z_{1,1},\cdots,z_{1,d})$ and $z_2=(z_{2,1},\cdots,z_{2,d})$.

For $i=1,\cdots, d$, we first introduce the following notations
\begin{align*}
I_i(H,t_2,t_1,z_2,z_1)&=e^{-\frac{1}{2} \E[z_{2,i}\cdot (X^{H,i}_{t_2}-X^{H,i}_{t_1})+z_{1,i}\cdot X^{H,i}_{t_1}]^2}\\
\widetilde{I}_i(H,t_2,t_1,z_2,z_1)&=e^{-\frac{1}{2} \E[z_{2,i}\cdot (\widetilde{X}^{H,i}_{t_2}-\widetilde{X}^{H,i}_{t_1})+z_{1,i}\cdot \widetilde{X}^{H,i}_{t_1}]^2}\\
K_i(\varepsilon,z_2, z_1)&=e^{-\frac{\varepsilon}{2}(z^2_{2,i}+z^2_{1,i})+\iota(z_{1,i}+z_{2,i})x_i}z^{k_i}_{2,i}z^{k_i}_{1,i}.
\end{align*}
Then we define
\begin{align*}
F_1(t_2,t_1, s_2,s_1, x_i)&=\frac{(-1)^{k_i}}{2\pi}\int_{\R^2}I_i(H_1,t_2,t_1,z_2,z_1+z_2)\widetilde{I}_i(H_2,s_2,s_1,z_2,z_1+z_2)K_i(\varepsilon,z_2, z_1) \, dz_{2,i}\, dz_{1,i}\\
F_2(t_2,t_1, s_2,s_1, x_i)&=\frac{(-1)^{k_i}}{2\pi}\int_{\R^2}I_i(H_1,t_2,t_1,z_2,z_1+z_2)\widetilde{I}_i(H_2,s_2,s_1,z_1,z_1+z_2)K_i(\varepsilon,z_2, z_1)\, dz_{2,i}\, dz_{1,i}.
\end{align*}
Now we can obtain that
\begin{align}  \label{e1}
\E\Big[ |L^{(\mathbf{k})}_{\varepsilon}(T,x)|^2\Big]&=\frac{2}{(2\pi)^d}\left[\int_{D}\prod^d_{i=1}F_1(t_2,t_1, s_2,s_1, x_i)\, dt\, ds+\int_{D}\prod^d_{i=1}F_2(t_2,t_1, s_2,s_1, x_i)\, dt\, ds\right] \nonumber \\
&=:\frac{2}{(2\pi)^d}(I_1(\varepsilon)+I_2(\varepsilon)),
\end{align}
where $D=\{0<s_1<t_1<T, 0<s_2<t_2<T\}$.

\noindent
{\bf Step 2.}  We estimate $\liminf\limits_{\varepsilon\downarrow 0}\frac{I_1(\varepsilon)}{h^{d,|{\bf k}|}_{H_1,H_2}(\varepsilon)}$. Note that
\[
I_1(\varepsilon)=\int_{D}\prod\limits^d_{i=1}F_1(t_2,t_1, s_2,s_1, x_i)\, dt\, ds.
\]
Making the change of variables $y_{2}=z_{2}$ and $y_1=z_1+z_2$ gives
\begin{align*}
F_1(t_2,t_1, s_2,s_1, x_i)&=\frac{(-1)^{k_i}}{2\pi}\int_{\R^2}I_i(H_1,t_2,t_1,y_2,y_1)\widetilde{I}_i(H_2,s_2,s_1,y_2,y_1)K_i(\varepsilon, y_2, y_1-y_2) \, dy_{2,i}\, dy_{1,i}.
\end{align*}
In order to calculate the above integral,  we set
\begin{align*}
a&=\E[(X^{H_1,1}_{t_2}-X^{H_1,1}_{t_1})^2]+\E[(\widetilde{X}^{H_2,1}_{s_2}-\widetilde{X}^{H_2,1}_{s_1})^2]\\
b&=\E[(X^{H_1,1}_{t_2}-X^{H_1,1}_{t_1})X^{H_1,1}_{t_1}]+\E[(\widetilde{X}^{H_2,1}_{s_2}-\widetilde{X}^{H_2,1}_{s_1})\widetilde{X}^{H_2,1}_{s_1}]\\
c&=\E[(X^{H_1,1}_{t_1})^2]+\E[(\widetilde{X}^{H_2,1}_{s_1})^2]\\
\Delta&=c+\varepsilon-\frac{(b-\varepsilon)^2}{a+2\varepsilon}.
\end{align*}

For any $(t_2,t_1,s_2,s_1)\in D$, it is easy to see that $a,b,c>0$ and
\begin{align*}
\Delta=\frac{ac+a\varepsilon+2c\varepsilon+\varepsilon^2-b^2+2b\varepsilon}{a+2\varepsilon}\geq \frac{c\varepsilon+\varepsilon^2}{a+2\varepsilon}>0,
\end{align*}
where we use $|b|\leq \sqrt{ac}\leq \frac{a+c}{2}$ in the first inequality.

By Lemma \ref{lma1}, $F_1(t_2,t_1, s_2,s_1, x_i)$ equals
\begin{align*}
\sum^{k_i}_{\ell=0}\sum^{k_i+\ell}_{m=0:\text{even}}\sum^{2k_i-m}_{n=0:\text{even}} c_{k_i,\ell,m,n}  (a+2\varepsilon)^{-\frac{m+1}{2}} (\frac{\varepsilon-b}{a+2\varepsilon})^{k_i+\ell-m}\Delta^{-(2k_i-m)-\frac{1-n}{2}} x^{2k_i-m-n}_ie^{-\frac{x^2_i}{2\Delta}},
\end{align*}
where $c_{k_i,\ell,m,n}=(-1)^{\ell-\frac{m+n}{2}} \binom{k_i}{\ell}  \binom{k_i+\ell} {m}\binom{2k_i-m} {n}(m-1)!! (n-1)!!$.

For any $\gamma>1$ and $(t_2,t_1,s_2,s_1)\in D$,  using the Cauchy-Schwartz inequality and properties {\bf (P1)} and {\bf (P2)}, we can show that
\begin{align*}
\big|\E[(X^{H_1,1}_{t_2}-X^{H_1,1}_{t_1})X^{H_1,1}_{t_1}]\big|\leq c_1\big[(\gamma^{H_1}+\gamma^{-H_1})a+\beta(\gamma)a^{\frac{1}{2}}\big],
\end{align*}
where $\gamma^{H_1}a$ comes from the case $\frac{1}{\gamma}<\frac{t_2-t_1}{t_1}<\gamma$,  $\gamma^{-H_1}a$ from the case $\frac{t_2-t_1}{t_1}\geq \gamma$, and $\beta(\gamma)a^{\frac{1}{2}}$ from the case $\frac{t_2-t_1}{t_1}\leq \frac{1}{\gamma}$. Similarly,
\begin{align*}
\big|\E[(\widetilde{X}^{H_2,1}_{s_2}-\widetilde{X}^{H_2,1}_{s_1})\widetilde{X}^{H_2,1}_{s_1}]\big|\leq c_2\big[(\gamma^{H_1}+\gamma^{-H_1})a+\beta(\gamma)a^{\frac{1}{2}}\big].
\end{align*}
Hence
\begin{align*}
\Big|\frac{\varepsilon-b}{a+2\varepsilon}\Big|
&\leq \frac{1}{2}+c_3\frac{(\gamma^{H_1}+\gamma^{-H_1}+\gamma^{H_2}+\gamma^{-H_2})a+\beta(\gamma)a^{\frac{1}{2}}}{a+2\varepsilon}\leq c_4\Big(\gamma^{H_1}+\gamma^{H_2}+\frac{\beta(\gamma)}{(a+2\varepsilon)^{\frac{1}{2}}}\Big).
\end{align*}

Let \begin{align} \label{dg}
D_{\gamma}=D\cap\Big\{0<\frac{t_2-t_1}{t_1}<\frac{T\wedge 1}{2\gamma}, \frac{T}{4}<t_1<\frac{T}{2}, 0<\frac{s_2-s_1}{s_1}<\frac{T\wedge 1}{2\gamma}, \frac{T}{4}<s_1<\frac{T}{2} \Big\}.
\end{align}
For any $(t_2,t_1,s_2,s_1)\in D_{\gamma}$, using the property {\bf (P2)}, we can show that
\begin{align} \label{det}
c_5(T^{2H_1}+T^{2H_2})\leq \Delta\leq c_6(T^{2H_1}+T^{2H_2})
\end{align}
provided that $\gamma$ is very large and $\varepsilon$ is very small.

Note that for any $\alpha>0$, the function $h(w)=\frac{1}{w^{\alpha}}e^{-\frac{1}{w}}\in(0,\alpha^{\alpha}e^{-\alpha}]$ when $w\in(0,+\infty)$. Choosing $\gamma$ large enough gives
\begin{align*}
\liminf_{\varepsilon\downarrow 0}\frac{I_1(\varepsilon)}{h^{d,|{\bf k}|}_{H_1,H_2}(\varepsilon)}
&\geq \liminf_{\varepsilon\downarrow 0} \frac{(1-c_4\beta(\gamma))^d}{h^{d,|{\bf k}|}_{H_1,H_2}(\varepsilon)}\int_{D} (a+2\varepsilon)^{-|{\bf k}|-\frac{d}{2}}\, \Delta^{-\frac{d}{2}}e^{-\frac{|x|^2}{2\Delta}}dt\, ds,
\end{align*}
where $0<c_4\beta(\gamma)<1$.

Using the inequality (\ref{det}), we can obtain that
\begin{align} \label{e2}
&\liminf_{\varepsilon\downarrow 0}\frac{I_1(\varepsilon)}{h^{d,|{\bf k}|}_{H_1,H_2}(\varepsilon)} \nonumber \\
&\geq c_7\liminf_{\varepsilon\downarrow 0} \frac{(1-c_4\beta(\gamma))^d}{h^{d,|{\bf k}|}_{H_1,H_2}(\varepsilon)}\int_{D_{\gamma}} (a+2\varepsilon)^{-|{\bf k}|-\frac{d}{2}}\, (T^{2H_1}+T^{2H_2})^{-\frac{d}{2}}e^{-\frac{|x|^2}{2c_5(T^{2H_1}+T^{2H_2})}}dt\, ds \nonumber \\
&\geq c_8e^{-\frac{|x|^2}{2c_5(T^{2H_1}+T^{2H_2})}} \liminf_{\varepsilon\downarrow 0} \frac{(1-c_4\beta(\gamma))^d}{h^{d,|{\bf k}|}_{H_1,H_2}(\varepsilon)}\int_{D_{\gamma}} ((t_2-t_1)^{2H_1}+(s_2-s_1)^{2H_2}+2\varepsilon)^{-|{\bf k}|-\frac{d}{2}}dt\, ds \nonumber \\
&\geq c_9e^{-\frac{|x|^2}{2c_5(T^{2H_1}+T^{2H_2})}},
\end{align}
where in the last inequality we use Lemma {\bf A.3.} in \cite{hx} and the property {\bf (P1)}.

\noindent
{\bf Step 3.}  We estimate $\liminf\limits_{\varepsilon\downarrow 0}\frac{I_2(\varepsilon)}{h^{d,|{\bf k}|}_{H_1,H_2}(\varepsilon)}$. Note that
\[
I_2(\varepsilon)=\int_{D}\prod\limits^d_{i=1}F_2(t_2,t_1, s_2,s_1, x_i)\, dt\, ds.
\]
Making the change of variables $y_{2}=z_{2}$ and $y_1=z_1+z_2$ gives
\begin{align*}
F_2(t_2,t_1, s_2,s_1, x_i)&=\frac{(-1)^{k_i}}{2\pi}\int_{\R^2}I_i(H_1,t_2,t_1,y_2,y_1)\widetilde{I}_i(H_2,s_2,s_1,y_1-y_2,y_1)K_i(\varepsilon, y_2, y_1-y_2) \, dy_{2,i}\, dy_{1,i}.
\end{align*}
In order to calculate the above integral,  we set
\begin{align*}
a_1&=\E[(X^{H_1,1}_{t_2}-X^{H_1,1}_{t_1})^2],\; a_2=\E[(\widetilde{X}^{H_2,1}_{s_2}-\widetilde{X}^{H_2,1}_{s_1})^2],\\
b_1&=\E[(X^{H_1,1}_{t_2}-X^{H_1,1}_{t_1})X^{H_1,1}_{t_1}],\; b_2=\E[(\widetilde{X}^{H_2,1}_{s_2}-\widetilde{X}^{H_2,1}_{s_1})\widetilde{X}^{H_2,1}_{s_1}]\\
e_1&=\E[(X^{H_1,1}_{t_1})^2],\; e_2=\E[(\widetilde{X}^{H_2,1}_{s_1})^2]\\
\Delta'&=e_1+e_2+a_2+2b_2+\varepsilon-\frac{(b_1-b_2-a_2-\varepsilon)^2}{a_1+a_2+2\varepsilon}.
\end{align*}
It is easy to see that, for any $(t_2,t_1,s_2,s_1)\in D$, $a_1, a_2, e_1,e_2>0$ and
\begin{align*}
\Delta'&\geq \frac{a_1e_2+a_1a_2+a_2e_1+2a_1b_2+2a_2b_1+2b_1b_2+\varepsilon(e_1+e_2)+\varepsilon^2}{a_1+a_2+2\varepsilon}\geq \frac{\varepsilon(e_1+e_2)+\varepsilon^2}{a_1+a_2+2\varepsilon}>0,
\end{align*}
where we use $|b_1|\leq \sqrt{a_1e_1}\leq \frac{a_1+e_1}{2}$ and $|b_2|\leq \sqrt{a_2e_2}\leq \frac{a_2+e_2}{2}$ in the first two inequalities.

By Lemma \ref{lma1}, $F_2(t_2,t_1, s_2,s_1, x_i)$ equals
\begin{align*}
\sum^{k_i}_{\ell=0}\sum^{k_i+\ell}_{m=0:\text{even}}\sum^{2k_i-m}_{n=0:\text{even}} c_{k_i,\ell,m,n} (a_1+a_2+2\varepsilon)^{-\frac{m+1}{2}} (\frac{\varepsilon+b_2+a_2-b_1}{a_1+a_2+2\varepsilon})^{k_i+\ell-m}(\Delta')^{-(2k_i-m)-\frac{1-n}{2}} x^{2k_i-m-n}e^{-\frac{x^2}{2\Delta'}},
\end{align*}
where $c_{k_i,\ell,m,n}=(-1)^{\ell-\frac{m+n}{2}} \binom{k_i}{\ell}  \binom{k_i+\ell} {m}\binom{2k_i-m} {n}(m-1)!! (n-1)!!$.

For any $(t_2,t_1,s_2,s_1)\in D$, using the Cauchy-Schwartz inequality and properties {\bf (P1)} and {\bf (P2)}, we can show that
\begin{align*}
\Big|\frac{\varepsilon+b_2+a_2-b_1}{a_1+a_2+2\varepsilon}\Big|
&\leq 1+\frac{|b_2-b_1|}{a_1+a_2+2\varepsilon}\leq c_{10}\Big(\gamma^{H_1}+\gamma^{H_2}+\frac{\beta(\gamma)}{(a_1+a_2+2\varepsilon)^{\frac{1}{2}}}\Big).
\end{align*}
Therefore,
\begin{align*}
\liminf_{\varepsilon\downarrow 0}\frac{I_2(\varepsilon)}{h^{d,|{\bf k}|}_{H_1,H_2}(\varepsilon)}
&\geq \liminf_{\varepsilon\downarrow 0} \frac{1}{h^{d,|{\bf k}|}_{H_1,H_2}(\varepsilon)}\int_{D} (a_1+a_2+2\varepsilon)^{-|{\bf k}|-\frac{d}{2}}\, \Delta^{-\frac{d}{2}}e^{-\frac{|x|^2}{2\Delta}}dt\, ds\\
&\qquad\qquad-c_{11}\beta(\gamma)\limsup_{\varepsilon\downarrow 0} \frac{1}{h^{d,|{\bf k}|}_{H_1,H_2}(\varepsilon)}\int_{D} (a_1+a_2+2\varepsilon)^{-|{\bf k}|-\frac{d}{2}}\, dt\, ds\\
&\geq -c_{12}\beta(\gamma)\limsup_{\varepsilon\downarrow 0} \frac{1}{h^{d,|{\bf k}|}_{H_1,H_2}(\varepsilon)}\int_{D} ((t_2-t_1)^{2H_1}+(s_2-s_1)^{2H_2}+2\varepsilon)^{-|{\bf k}|-\frac{d}{2}}\,dt\, ds\\
&\geq -c_{13}\beta(\gamma).
\end{align*}
Letting $\gamma\uparrow+\infty$ gives
\begin{align} \label{e3}
\liminf_{\varepsilon\downarrow 0}\frac{I_2(\varepsilon)}{h^{d,|{\bf k}|}_{H_1,H_2}(\varepsilon)}\geq 0.
\end{align}

\noindent
{\bf Step 3.}  Combining \eref{e1}, \eref{e2} and \eref{e3} gives
\begin{align*}
\liminf_{\varepsilon\downarrow 0}\frac{\E[|L^{(\mathbf{k})}_{\varepsilon}(T,x)|^2]}{h^{d,|{\bf k}|}_{H_1,H_2}(\varepsilon)}
&\geq \liminf_{\varepsilon\downarrow 0}\frac{I_1(\varepsilon)}{h^{d,|{\bf k}|}_{H_1,H_2}(\varepsilon)}+\liminf_{\varepsilon\downarrow 0}\frac{I_2(\varepsilon)}{h^{d,|{\bf k}|}_{H_1,H_2}(\varepsilon)}\geq c_{14}e^{-\frac{|x|^2}{2c_{5}(T^{2H_1}+T^{2H_2})}}.
\end{align*}
This completes the proof.
\end{proof}

\section{Proof of Theorem \ref{thm2}}

In this section, we give the proof of Theorem \ref{thm2}.

\begin{proof}
By Lemma \ref{lma3},
\begin{align*}
\E[|L^{(\mathbf{k})}_{\varepsilon}(T,0)|^2]
&=\frac{(-1)^{|\mathbf{k}|}}{(2\pi)^{2d}}\int_{[0,T]^4}\int_{\R^{2d}} e^{-\frac{1}{2}\big[\E(z_2\cdot X^{H_1}_{t_2}+z_1\cdot X^{H_1}_{t_1})^2+\E(z_2\cdot \widetilde{X}^{H_2}_{s_2}+z_1\cdot \widetilde{X}^{H_2}_{s_1})^2\big]}\\
&\qquad\qquad \times e^{-\frac{\varepsilon}{2}(|z_2|^2+|z_1|^2)} \prod\limits^d_{i=1}z^{k_i}_{2,i} \prod\limits^d_{i=1}z^{k_i}_{1,i}\, dz_2\, dz_1\, dt\, ds\\
&=\frac{1}{(2\pi)^d}\int_{[0,T]^4} \prod^d_{i=1}\Big(\sum^{k_i}_{\ell=0, \text{even}}    \frac{c_{k_i,\ell} \,b^{k_i-\ell}}{((a+\varepsilon)(c+\varepsilon)-b^2)^{\frac{2k_i-\ell+1}{2}}}
\Big)dt\, ds,
\end{align*}
where $a=\E[(X^{H_1,1}_{t_2})^2]+\E[(\widetilde{X}^{H_2,1}_{s_2})^2]$, $b=\E[X^{H_1,1}_{t_2}X^{H_1,1}_{t_1}]+\E[\widetilde{X}^{H_2,1}_{s_2}\widetilde{X}^{H_2,1}_{s_1}]$ and $c=\E[(\widetilde{X}^{H_2,1}_{s_1})^2]+\E[(X^{H_1,1}_{t_1})^2]$. According to the assumption (i) $|\mathbf{k}|$ is even or  (ii) $\E[X^{H_1,1}_{t}X^{H_1,1}_{s}]\geq 0$ and  $\E[\widetilde{X}^{H_2,1}_{t}\widetilde{X}^{H_2,1}_{s}]\geq 0$ for any  $0<s,t<T$,
\begin{align*}
\E[|L^{(\mathbf{k})}_{\varepsilon}(T,0)|^2]
&\geq \frac{1}{(2\pi)^d}\prod^d_{i=1} c_{k_i,0}\int_{[0,T]^4} \frac{b^{|\mathbf{k}|}}{\Big((a+\varepsilon)(c+\varepsilon)-b^2\big)^{\frac{2|\mathbf{k}|+d}{2}}}\, dt\, ds.
\end{align*}

Recall the definition of $D_{\gamma}$ in \eref{dg}. For any $(t_2,t_1,s_2,s_1)\in D_{\gamma}$ with $\gamma$ large enough, by the property {\bf (P1)}  and the Cauchy Schwartz inequality,
\begin{align*}
b&=\E[(X^{H_1,1}_{t_1})^2]+\E[(\widetilde{X}^{H_2,1}_{s_1})^2]+\E[(X^{H_1,1}_{t_2}-X^{H_1,1}_{t_1})X^{H_1,1}_{t_1}]+\E[(\widetilde{X}^{H_2,1}_{s_2}-\widetilde{X}^{H_2,1}_{s_1})\widetilde{X}^{H_2,1}_{s_1}]\\
&\geq c_1(t^{2H_1}_1+s^{2H_2}_1-\frac{t^{2H_1}_1}{\gamma^{H_1}}-\frac{s^{2H_1}_1}{\gamma^{H_2}})\\
&\geq c_2(t^{2H_1}_1+s^{2H_2}_1).
\end{align*}
Let $a_{22}=\E[(X^{H_1,1}_{t_2}-X^{H_1,1}_{t_1})^2]+\E[(\widetilde{X}^{H_2,1}_{s_2}-\widetilde{X}^{H_2,1}_{s_1})^2]$, $a_{11}=\E[(\widetilde{X}^{H_2,1}_{s_1})^2]+\E[(X^{H_1,1}_{t_1})^2]$ and $a_{12}=\E[(X^{H_1,1}_{t_2}-X^{H_1,1}_{t_1})X^{H_1,1}_{t_1}]+\E[(\widetilde{X}^{H_2,1}_{s_2}-\widetilde{X}^{H_2,1}_{s_1})\widetilde{X}^{H_2,1}_{s_1}]$. Then
\begin{align*}
0<(a+\varepsilon)(c+\varepsilon)-b^2
&=(a_{22}+2a_{12}+a_{11}+\varepsilon)(a_{11}+\varepsilon)-(a_{12}+a_{11})^2\leq 2(a_{22}+\varepsilon)(a_{11}+\varepsilon).
\end{align*}
Therefore, for $\gamma$ large enough,
\begin{align*}
&\liminf_{\varepsilon\downarrow 0}\frac{\E[|L^{(\mathbf{k})}_{\varepsilon}(T,0)|^2]}{h^{d,|{\bf k}|}_{H_1,H_2}(\varepsilon)}\\
&\geq \liminf_{\varepsilon\downarrow 0}\frac{c_3}{h^{d,|{\bf k}|}_{H_1,H_2}(\varepsilon)}  \int_{[0,T]^4} \frac{b^{|\mathbf{k}|}}{\Big((a+\varepsilon)(c+\varepsilon)-b^2\big)^{\frac{2|\mathbf{k}|+d}{2}}}\, dt\, ds\\
&\geq \liminf_{\varepsilon\downarrow 0}\frac{c_4}{h^{d,|{\bf k}|}_{H_1,H_2}(\varepsilon)}  \int_{D_{\gamma}} \frac{(t^{2H_1}_1+s^{2H_2}_1)^{|\mathbf{k}|}}{((t_2-t_1)^{2H_1}+(s_2-s_1)^{2H_2}+\varepsilon)^{|\mathbf{k}|+\frac{d}{2}}(t^{2H_1}_1+s^{2H_2}_1+\varepsilon)^{|\mathbf{k}|+\frac{d}{2}}}\, dt\, ds\\
&\geq c_5,
\end{align*}
where in the last inequality we use Lemma {\bf A.3.} in \cite{hx} and the property {\bf (P1)}.
\end{proof}

$\begin{array}{cc}
\begin{minipage}[t]{1\textwidth}
{\bf Minhao Hong}\\
School of Statistics, East China Normal University, Shanghai 200262, China \\
\texttt{hongmhecnu@foxmail.com}
\end{minipage}
\hfill
\end{array}$

$\begin{array}{cc}
\begin{minipage}[t]{1\textwidth}
{\bf Fangjun Xu}\\
Key Laboratory of Advanced Theory and Application in Statistics and Data Science - MOE, School of Statistics, East China Normal University, Shanghai, 200062, China \\
NYU-ECNU Institute of Mathematical Sciences at NYU Shanghai, 3663 Zhongshan Road North, Shanghai, 200062, China\\
\texttt{fangjunxu@gmail.com, fjxu@finance.ecnu.edu.cn}
\end{minipage}
\hfill
\end{array}$

\end{document}